\documentclass{article}

\usepackage[utf8]{inputenc} 

\usepackage{url}
\usepackage[colorlinks=true,breaklinks=true,bookmarks=true,urlcolor=blue,
     citecolor=blue,linkcolor=blue,bookmarksopen=false,draft=false]{hyperref}
\usepackage{graphicx}
\usepackage{enumerate}
\usepackage{multirow,bigdelim}
\usepackage{mathtools}
\usepackage{amsmath,amsfonts,amssymb}
\usepackage{amsthm}
\usepackage{soul} 
\usepackage{color}
\usepackage{array}
\usepackage{complexity}
\usepackage[normalem]{ulem}
\newcommand{\interior}{\mathrm{int}\,}
\newcommand{\reInt}{\mathrm{ri}\,}

\newcommand{\closure}{\mathrm{cl}\,}
\newcommand{\dirCone}{\mathrm{dir}\,}
\newcommand{\tanCone}{\mathrm{tanCone}\,}

\newcommand{\norm}[1]{\lVert{#1}\rVert}
\newcommand{\inProd}[2]{\langle #1 , #2 \rangle }

\title{A structural geometrical analysis of weakly infeasible SDPs  \\ \ }

\author{Bruno F. Louren\c{c}o
\thanks{
Department of Mathematical and Computing Sciences,
Tokyo Institute of Technology,
2-12-1-W8-41 Ookayama, Meguro-ku, Tokyo 152-8552, Japan. (E-mail: flourenco.b.aa@m.titech.ac.jp)}
        \and
        Masakazu Muramatsu\thanks{
                     Department of Computer Science, The University of Electro-Communications 1-5-1 Chofugaoka, Chofu-shi, Tokyo, 182-8585 Japan. (E-mail: muramatu@cs.uec.ac.jp)
                  }
         \and       
                Takashi Tsuchiya
\thanks{
National Graduate Institute for Policy Studies 7-22-1 Roppongi, Minato-ku, Tokyo 106-8677, Japan. (E-mail: tsuchiya@grips.ac.jp) \newline M. Muramatsu and T. Tsuchiya are supported in part with Grant-in-Aid for Scientific Research (B)24310112
                  }
        }

\date{November 2013 (Revised: July 2015) }

\newtheorem{definition}{Definition}

\newtheorem{proposition}[definition]{Proposition}
\newtheorem{example}[definition]{Example}
\newtheorem{corollary}[definition]{Corollary}

\newtheorem{theorem}[definition]{Theorem}

\begin{document}
\maketitle
\begin{abstract}
In this article, we present a geometric theoretical analysis of
semidefinite feasibility problems (SDFPs). This is done by
decomposing a SDFP into smaller problems, in a way
that preserves most feasibility properties of the original problem. 
With this technique, we develop a detailed analysis of weakly infeasible SDFPs
to understand clearly and systematically how weak
infeasibility arises in semidefinite programming. 
In particular, we show that for a weakly infeasible problem over $n\times n$ matrices, 
at most $n-1$ directions are required to approach the positive semidefinite cone.
We also present a discussion on feasibility certificates for SDFPs and related 
complexity results.
\end{abstract}

\renewcommand{\topfraction}{0.9}

\section{Introduction.}
In this paper, we deal with the following semidefinite feasibility problem
\begin{equation}
\max \; 0 \ \ \text{s.t.}\ \; x \in (L+c)\cap K_n, \label{sdpf}
\end{equation}
where $L \subseteq \mathbb{S}_n$ is a vector subspace and 
$c \in \mathbb{S}_n$. By $\mathbb{S}_n$ we denote the linear space 
of $n\times n$ real symmetric matrices and $K_{n} \subseteq \mathbb{S}_{n}$ denotes the cone 
of $n\times n$ positive semidefinite matrices. We denote the problem \eqref{sdpf} by $(K_n,L,c)$.

It is known that every instance of a semidefinite program 
falls into one of the following four statuses:
\begin{itemize}
	\item \emph{Strongly feasible}: $(L+c)\cap \interior (K_n) \neq \emptyset$, where 
	$\interior (K_n)$ denotes the interior of $K_n$.
	\item \emph{Weakly feasible}: $(L+c)\cap \interior (K_n) = \emptyset$, but $(L+c)\cap K_n \neq \emptyset$.
	\item \emph{Weakly infeasible}: $(L+c)\cap K_n = \emptyset$ and $\mathrm{dist}(K,L+c) = 0$.
	\item \emph{Strongly infeasible}: $(L+c)\cap K_n = \emptyset$ and $\mathrm{dist}(K,L+c) > 0$.
\end{itemize}
 
Among the four feasibility statuses, all but weak infeasibility afford simple finite certificates: 
an interior-feasible solution, a pair consisting of a feasible solution and a vector which is normal to a separating hyperplane, 
and a dual improving direction for strong feasibility, weak feasibility and strong infeasibility respectively.  The last one is sometimes called a Farkas-type certificate, and
plays an important role in optimization theory.  However,  it is not evident whether weak infeasibility affords such a finite certificate. 

By ``finite certificate'' we mean a finite sequence in some finite dimensional vector space.
In this paper, we focus on the structural analysis of weak infeasibility in semidefinite programming and 
develop a procedure which distinguishes the four statuses. 
We also obtain a finite certificate for weak infeasibility. 
But we emphasize that the main feature of our approach is concreteness in analyzing weak infeasibility.

In view of finite certificates, we  mention that
it is possible to obtain a finite and polynomially bounded certificate of weak infeasibility by using
Ramana's extended Lagrangian dual \cite{Ramana95anexact}.  This result is based on the fact that 

\medskip
{\it 
$(K_n,L,c)$ is weakly infeasible if and only if it is \emph{infeasible} and \emph{not strongly infeasible}
}

\medskip\noindent 
as we will discuss in Section 2. Ramana developed a generalized Farkas' Lemma 
for SDP which holds without any assumptions. Since \emph{infeasibility} and \emph{not strong infeasibility} 
have finite certificates, the same is true for weak infeasibility. 
As this argument has an existential flavour, it is not so clear the implications for 
the structure of the problem. 
In this paper, we study weak infeasibility in semidefinite programming  from a more constructive  
point of view to answer, for instance, the following basic question:

\medskip{\it
Given a weakly infeasible SDFP, how can we generate a sequence 
$\{u^{(i)}|\ u^{(i)}\in L + c, i = 1,\ldots \infty\}$ such that $\lim_{i\rightarrow\infty}{\rm dist}(u^{(i)}, K_n) =0$?
}

\medskip \noindent
Due to the fact that the distance between $K_n$ and $L+c$ is 
zero, we readily see that there exists a nonzero element $a$ in $K_n\cap L$.
However, it is not clear how $a$ is related to the weak infeasibility of $(K_n,L,c)$. 
Since the problem is infeasible, ${\rm dist}(ta + b, K_n)> 0$ for any $t > 0$ and $b\in L+c$.  
It would be natural to ask what to happen as $t$ goes to infinity. 
Can $\lim_{t\rightarrow \infty} {\rm dist}(ta + b, K_n)=\infty$ or a finite 
nonzero value, or zero? If we cannot find any $b\in L+c$ such that
$\lim_{t\rightarrow \infty} {\rm dist}(ta + b, K_n)= 0$ holds, how $a$ can be used to construct points 
close to the cone?  

We will show that $a$ alone is not enough to generate such a sequence, but
$(n-1)$ directions including $a$ are sufficient (with an appropriate choice
of $b$), whenever the problem is weakly infeasible.
In other words, if $(K_n, L, c)$ is weakly infeasible then there exists a $(n-1)$ dimensional affine subspace 
${\cal F} \subseteq L + c$ such that ${\cal F}\cap K_n = \emptyset$ but ${\rm dist}({\cal F}, K_n)=0$.
This result is a bit surprising to us, because, in general, 
if $K$ is a closed convex cone and $(K,L,c)$ is weakly infeasible, then the number of
directions necessary to approach the cone could be as large as the dimension of $L$, which
could be up to $ (\frac{n(n+1)}{2} - 1)$ in our context.

The proof is done by constructing a set of directions in $L$ which we call {\it hyper feasible partition}.
These direction are obtained recursively starting from a nonzero element in $K_n\cap L$.  
An important feature of this set is that, even though each direction is not necessarily positive (semi)definite,
we can always find a positive linear combination which is almost positive semidefinite
(the minimum eigenvalue can be made to be arbitrarily close to zero).  
The introduction of hyper feasible partitions is another main contribution of this paper
and they provide a new insight in the analysis of ill-conditioned semidefinite programs.

One possible application of our results is as follows. Consider the 
following SDP
\begin{equation*}
\max \; \inProd{b}{x} \ \ 
\text{s.t.} \; x \in (L+c)\cap K_n \tag{P}\label{sdp_primal}, 
\end{equation*}
and suppose that the optimal value $b^*$ is finite  but not attained. The set 
$\{x \in L+c \mid  \inProd{b}{x} = b^*\}$ is non-empty and is also an affine space.
Denoting by $\widetilde {L}$ the underlying vector space and letting $\widetilde 
c$ be any point which belongs to the affine space, we have that $(K_n,\widetilde{L}, \widetilde{c})$ is 
weakly infeasible.  Indeed such problems arise in many applications in semidefinite programming
including control theory and polynomial optimization \cite{waki_how_2012}.

The main tool we use is a simple decomposition result (Theorem \ref{theo_decomp}), 
which implies that 
some semidefinite feasibility problems (SDFPs) can be decomposed into smaller subproblems in a way that the feasibility 
properties are mostly preserved. We also discuss two procedures for analyzing feasibility 
problems, a \emph{forward procedure} (\textbf{FP}) and a \emph{backward procedure} (\textbf{BP}). In particular, 
\textbf{BP} can distinguish the 4 different feasibility statuses in a systematic manner.

We review related previous works.
The existence of weak infeasibility/feasibility and finite duality gap is one of the main 
difficulties in semidefinite programming.  These situations may occur in the absence of 
interior-feasible solutions to the primal and/or dual.  
Two possible techniques to recover interior-feasibility
by reducing the feasible region of the problem
or by expanding the feasible region of its dual 
counter-part are the facial reduction algorithm (FRA) and the conic expansion approach
(CEA), respectively.
FRA was developed by Borwein and Wolkowicz \cite{borwein_facial_1981} for problems 
more general than conic programming, whereas  CEA was developed by Luo, Sturm and Zhang \cite{Luo96dualityand} for conic programming.

In the earlier stages of research of semidefinite programming, 
Ramana \cite{Ramana95anexact} developed an extended Lagrange-Slater dual  (ELSD) that has 
no duality gap.  ELSD has the remarkable feature that the size of the extended 
problem is bounded by a polynomial in terms of the size of the original problem.
In \cite{ramana_strong_1997}, Ramana, Tun\c{c}el and Wolkowicz demonstrated that ELSD can be interpreted as 
a facial reduction problem, however, we should note that in the original FRA, the size 
of the problem is not polynomially bounded, see also \cite{pataki_strong_2013}. 
In \cite{polik07b}, Polik and Terlaky provided 
strong duals for conic programming over symmetric cones.
Recently, Klep and Schweighofer developed another dual
based on real algebraic geometry where the strong duality holds without any constraint 
qualification \cite{klep_exact_2013}.  Like ELSD, their dual is just represented in terms of the data of 
original problem and the size of the dual is bounded by a polynomial in terms of the size of the original problem.
Complexity of SDFP is yet a subtle issue.  This topic was studied extensively by
Porkolab and Khachiyan \cite{Porkolab_Khachiyan_97}. 

Waki and Muramatsu \cite{article_waki_muramatsu}  considered a FRA for conic programming and showed that 
FRA can be regarded as a dual version of CEA. See an excellent review by Pataki \cite{pataki_strong_2013}
for FRA, where he points out the relation between facial reduction and extended duals. 
Pataki also found  that all ill-conditioned semidefinite programs can be reduced 
to a common $2\times 2$ semidefinite program \cite{pataki_bad_sdps}.
Finally, we mention that Waki showed that weakly infeasible instances can be obtained
from semidefinite relaxation of polynomial optimization problems \cite{waki_how_2012}.

The problem of weak infeasibility is closely related to closedness of the image of $K_n$ by a certain linear map.
A comprehensive treatment of the subject was given by Pataki \cite{pataki_closedness_2007}.
We will discuss the connection between Pataki's results and weak infeasibility in Section 2.

This paper is organized as follows. In Section \ref{sec:certificates}, we 
discuss certificates for the different feasibility statuses and point the connections
to previous works. 
In Section \ref{sec:decomp} we present Theorem \ref{theo_decomp} and discuss 
how certain SDFPs can be broken in smaller problems. We also prove the bound 
$n-1$ for the number of the directions needed to approach $K_n$.
In Section \ref{sec:backward_procedure}, a procedure to distinguish between 
the $4$ different feasibility statues is given. Section \ref{sec:conc} summarizes 
this work.

\section{Characterization of different feasibility statuses}\label{sec:certificates}

In this section, we review the characterization of different feasibility statuses of semidefinite programs
with emphasis on weak infeasibility. 
\subsection{Certificates and \NP\ class in the Blum-Shub-Smale model}

Our main interest is on \emph{finite certificates} and computational complexity. 
The model of computation we use is 
the Blum-Shub-Smale model (BSS model) \cite{blum_complexity_1997} of real computation.
The main aspects are that we do not care about the bit length of 
a real number, we can evaluate any rational function over $\mathbb{R}$
and the machine can deviate the flow of execution by evaluating a linear inequality.
``Finite'' in this context means that the certificates are composed of a finite 
number of vectors contained in some  finite dimensional vector space. The length of 
the certificate is then the total number of coordinates among all the vectors it contains. 
It is also required 
that a verifier procedure exists. Such a procedure receives as input 
the problem and the certificate and attest that the certificate is indeed 
valid in a finite amount of time.  
If a decision problem admits a finite certificate with a verification procedure
such that the length of the former and time complexity of the latter are polynomials in terms of the size of the problem
then it is in \NP\ under the BSS model. The main decision problem we are interested in is: 
\emph{given $(K_n,L,c)$, what is its 
feasibility status?}\footnote{Strictly speaking, a decision problem should have 
``yes'' or ``no'' as answers, but in our case the possible answers are 
strong/weakly feasible, strong/weakly infeasible. 
We could have broken down the decision problem in 4 different decision 
problems having ``yes'' or ``no'' as answers. We did not do so, because we wanted to treat 
them in a unified manner in our procedure \textbf{BP}.}

Suppose we have an algorithm for a decision problem which employs oracles 
for some problems, for instance, returning a feasible solution to a SDP. (This is a typical situation
in the literature when talking about regularization procedures.) 
In such a situation, if we want to show that the decision problem is in \NP, we can do the following.
First, we prove that correctness of the output of each oracle can be verified in 
polynomial time (with respect the size of the problem).
Then, we evaluate the time complexity of the algorithm assuming that the cost for each call 
of the oracle is one. If the running time of the algorithm is bounded by a polynomial in the size of the problem,
then the decision problem is in $\NP$.
The set of outputs given by the oracles can be used as a certificate and the algorithm 
itself acts as a verifier procedure.

Throughout this paper, we assume that $L$ is represented as the set of solution of the system of
linear equations, where the coefficients and left hand side is explicitly given.  We also note that
checking positive semidefiniteness of a symmetric matrix can be done in polynomial time by using 
a variant of $LDL^T$ decomposition.

\subsection{Characterization of feasibility statuses}

We start with the following proposition which characterizes strong feasibility, 
weak feasibility and strong infeasibility.

\begin{proposition}\label{prop_certificate}
Let $L$ be a subspace of $\mathbb{S}_n$ and $c \in \mathbb{S}_n$ then
$(K_n,L,c)$ is:
\begin{enumerate}
\item \emph{Strongly feasible}, if and only if there is $ x \in L +c$ such that $x$ is positive definite.
\item \emph{Weakly feasible} if and only if there is 
\begin{enumerate}[i.]
	\item $x \in L +c$ such that $x$ is positive semidefinite 
	\item $y \in L^\perp \cap K_n$ such that  $y \neq 0$, $\inProd{y}{c} = 0$.
\end{enumerate} 
\item \emph{Strongly infeasible}  if and only if 
$c\neq 0$ and $(K_n, L_{\rm SI}, c_{\rm SI})$ is feasible, where 
$L_{\rm SI} = L^\perp \cap {c}^\perp$, and  $c_{\rm SI} = -\frac{c}{\norm{c}^2}$.
\end{enumerate}
\end{proposition}
\begin{proof}
Item $i$ is immediate. Items $ii.$ and $iii.$ follow easily from Theorem 11.3 and 11.4 of \cite{rockafellar}.
Item $ii.$ correspond to the situation where $K_n$ and $L+c$ can be properly separated but still 
a feasible point exists and 
item $iii.$ to the case where they can be strongly separated. Also, for a proof of $iii.$ see, for instance,
Lemma 5 of \cite{Luo97dualityresults}.
\end{proof}
Proposition \ref{prop_certificate} already implies that deciding 
deciding strong feasibility, weak feasibility and 
strong infeasibility are in $\NP$, in the BSS model. In addition, 
$(K_n,L,c)$ is not strongly feasible if and only if item $2.ii$ holds (but not 
necessarily $2.i$). This means that deciding strong feasibility lies in $\coNP$ as well\footnote{In this case, 
the decision problem has either ``yes'' or ``no'' as possible answers, so it makes sense to talk about $\coNP$.}.

In Proposition \ref{prop_certificate}, weak infeasibility is  
absent. When proving weak infeasibility, it is necessary to show that 
the distance between $K_n$ and $L+c$ is $0$. The obvious way 
is to produce a sequence 
$\{x_k\} \in L+c$ such that $\lim _{k \to +\infty} \text{dist}(x_k,K_n) = 0$.  
In \cite{Luo96dualityand} it was shown that $(K_n,L,c)$ is weakly infeasible if and only 
if there is no dual improving direction \emph{and} there is a dual improving sequence (see Lemma 6 and Table 1 in \cite{Luo96dualityand}). 
But this is not a finite certificate of weak infeasibility.

A finite certificate for weak infeasibility can be obtained by using 
Ramana's results on an extended Lagrangian dual for semidefinite programming \cite{Ramana95anexact}.  
Ramana's dual has a number of key properties:
it is written explicitly in terms of problem data, it has no duality gap and 
the optimal value is always attained when finite. With his dual, it was possible 
to develop an exact Farkas-type lemma for semidefinite programming. In Theorem 
19 of \cite{Ramana95anexact}, 
he constructed another SDFP $\mathcal{RD}(K_n,L,c)$ for which the following holds
without any regularity conditions:

\medskip
\ \ \ \ \ \ \ \ {\it $(K_n,L,c)$ is feasible if and only if $\mathcal{RD}(K_n,L,c)$ is infeasible.}
\medskip

\noindent
Furthermore, the size of $\mathcal{RD}(K_n,L,c)$ is 
bounded by a polynomial that depends only on the size of the system $(K_n,L,c)$.
Based on this strong result, we obtain a finite certificate of weak infeasibility as
in the following proposition:

\begin{proposition}\label{col_weak_inf}
We have the following:
\begin{enumerate}[$i.$]
\item $(K_n,L,c)$ is weakly infeasible $\Leftrightarrow$ $c \neq 0$, $\mathcal{RD}(K_n,L,c)$ and 
$\mathcal{RD}(K_n,{L}_{\rm SI},{c}_{\rm SI})$ are
feasible.
\item The problem of deciding whether a given $(K_n,L,c)$ is weakly infeasible is 
in $\NP \cap \coNP$ in the BSS model.
\end{enumerate}

\end{proposition}
\begin{proof}
A feasible solution to $\mathcal{RD}(K_n,L,c)$ attests the infeasibility 
of $(K_n,L,c)$. As $\mathcal{RD}(K_n,L,c)$ has polynomial size, it is possible 
to check that a point is indeed a solution to it in polynomial time.

Note that a problem is weakly infeasible if and only if it is infeasible 
and is not strongly infeasible. 
Due to Proposition 1, we have that 
$(K_n,L,c)$ is not strongly infeasible if and only if $c = 0$ or $c\neq 0$ and $\mathcal{RD}(K_n,L_{\rm SI}, c_{\rm SI})$ is 
feasible. Hence, feasible solutions to $ \mathcal{RD}(K_n,L,c)$ and $\mathcal{RD}(K_n,L_{\rm SI},c_{\rm SI})$ can 
be used together as a certificate for  weak infeasibility. Such a certificate can be checked in polynomial time,
hence the problem is in $\NP$.

Now, $c = 0$, a solution to $(K_n,L_{\rm SI}, c_{\rm SI})$ \emph{or} to $(K_n,L,c)$ can be used 
to certify that a system is not weakly infeasible. This shows that deciding weak infeasibility 
is indeed in $\coNP$.
\end{proof}

The important point in the argument above is 
having both a certificate of infeasibility for the original system and a certificate 
of infeasibility for the system $(K_n, L_{\rm SI}, c_{\rm SI})$. 
Any method of obtaining finite certificates of 
infeasibility can be used in place of $\mathcal{RD}$, as long as it takes polynomial time 
to verify them.  See the comments after  Theorem 3.5 in Sturm \cite{sturm_error_2000} and also Theorem 7.5.1 of \cite{sturm_handbook} for 
another certificate of infeasibility.
Klep and Schweighofer \cite{klep_exact_2013} also developed certificates for infeasibility and 
 a hierarchy of infeasibility in which $0$-infeasibility corresponds to 
strong infeasibility and $k$-infeasibility to weak infeasibility, when $k > 0$.
Liu and Pataki \cite{pataki_liu_2014} also introduced an infeasibility 
certificate for semidefinite programming. They defined what is a reformulation 
of a feasibility system and showed that $(K_n,L,c)$ is infeasible if and only 
if it admits a reformulation that converts the systems to a special format, see Theorem 
1 therein.

We mention a few more related works on weak infeasibility.
The feasibility problem $(K_n,L,c)$ is weakly infeasible if and only if 
$c \in \closure (K_n+L) \setminus (K_n+L)$, where $\closure$ denotes the closure 
operator. 
Hence, a \emph{necessary} condition for weak infeasibility is that $K_n+L$ fails 
to be closed.  This problem is closely related to closedness of the image of 
$K_n$ by a linear map which is the problem analyzed in detail by Pataki 
\cite{pataki_closedness_2007}.  
Theorem 1.1 in \cite{pataki_closedness_2007} provides a 
necessary and sufficient condition for the failure of closedness of $K_n+L$.  
Pataki's result implies that there is some  $c \in \mathbb{S}_n$ such that $(K_n,L,c)$ is weakly infeasible if and only if  
$L^\perp \cap (\closure \dirCone(x,K_n) \setminus \dirCone (x,K_n)) \neq \emptyset$, where 
$x$ belongs to the relative interior of $L\cap K_n$ and $\dirCone (x,K_n)$ is the 
cone of feasible directions at $x$.
This tells us whether $K_n$ and $L$ can accommodate a weakly infeasible problem.
If it is indeed possible, Corollary 3.1 of \cite{pataki_closedness_2007} shows how 
to find an appropriate $c$.
Bonnans and Shapiro \cite{bonnans_perturbation_2000} also discussed generation of 
weakly infeasible semidefinite programming problems. As a by-product of the proof of Proposition 2.193 therein, it is 
shown how to construct weakly infeasible problems.

In \cite{pataki_bad_sdps}, Pataki introduced the notion of \emph{well-behaved} system. 
$(K_n,L,c)$ is said to be well-behaved if for all $b \in \mathbb{S}_n$, the optimal value of \eqref{sdp_primal}  
and of its dual are the same and the dual  is attained whenever it is finite. A SDP which 
is not well-behaved is said to be \emph{badly-behaved}. 
Pataki showed that badly-behaved SDPs can be put 
into a special shape, see Theorem 6 in \cite{pataki_bad_sdps}. Then, 
a necessary condition for weak infeasibility is that the homogenized system 
$(K_n,\widetilde {L},0)$ be badly-behaved, where $\widetilde {L}$ is spanned by $L$ and $c$.
See the comments before Section 4 in \cite{pataki_bad_sdps}.

\section{A decomposition result.}\label{sec:decomp}

In this section, we develop a key decomposition result.  
Given an SDFP, we show how to  construct a smaller dimensional SDFP which 
preserves most of the feasibility properties.

\subsection{Preliminaries}

First 
we introduce the notation.
If $\mathcal{C},\mathcal{D}$ are subsets of some real 
space, we write $\mathrm{dist}(\mathcal{C},\mathcal{D}) = \inf \{\|x-y\| \mid x \in \mathcal{C}, y \in \mathcal{D}\}$, where $\| \cdot \|$ is the Euclidean norm or 
the Frobenius norm, in the case of subsets of $\mathbb{S}_{n}$. By $\interior (\mathcal{C})$ and 
$\reInt(\mathcal{C})$ we denote the interior and the relative interior of $\mathcal{C}$, respectively.
We use $I_{n}$ to denote the $n\times n$ identity matrix. Given $(K_n,L,c)$ and a 
matrix $A \in K_n\cap L$ with rank $k$, we will call $A$ a \emph{hyper feasible direction} of 
rank $k$. We remark that when $(K_n,L,c)$ is feasible, $A$ is also a recession direction of 
the feasible region.

Let $x$ be a $n\times n$ matrix, and
$0\leq k \leq n$. We denote by $\pi_k(x)$.
the upper left $k\times k$ principal submatrix of $x$.
For instance, if
\[
x = \left(\begin{array}{ccc} 1 & 2 & 3 \\ 2 & 4 & 5 \\ 3 & 5 & 6 \end{array}\right), 
\]
then,
\[
\pi_{2}(x) = \left(\begin{array}{cc}1&2\\2&4\end{array}\right).
\]
We define the subproblem $\pi_k(K_n, L, c)$
of $(K_n, L, c)$ to be
\[ 
{\rm find}\ u \in \pi_k(L+c),\ \ \ u \succeq 0.
\]
In other words, it is the feasibility problem $(\pi_k(K_n), \pi_k(L), \pi_k(c)) $.
We denote by ${\overline \pi}_k(x)$, the lower right $(n-k)\times (n-k)$ principal 
submatrix. In the example above, we have 
$
\pi_{\overline{2}}(x) = 6.
$
In a similar manner, we write $\overline\pi_k(K_n, L, c)$ for 
the feasibility problem $(\overline \pi_k(K_n), \overline \pi_k(L), \overline \pi_k(c)) $.
We remark that $\pi _n(x) = \overline \pi _0(x) = x$ and we define $\pi _0 (x) = \overline \pi _n(x) = 0$.

The proposition belows summarizes the properties of the Schur Complement.
For proofs, see Theorem 7.7.6 of \cite{matrix_analysis}.
\begin{proposition}[Schur Complement]
Suppose $M =\bigl(\begin{smallmatrix}
A&B\\ B^T&C
\end{smallmatrix} \bigr)$ is a symmetric matrix divided in blocks in a way that 
$A$ is positive definite, then:
\begin{itemize}
	\item $M$ is positive definite if and only if $C - B^TA^{-1}B$ is.
	\item $M$ is positive semidefinite if and only if $C - B^TA^{-1}B$ is.
\end{itemize}
\end{proposition}

The properties of a semidefinite program are 
not changed when a congruence transformation is applied, i.e, for any non-singular 
matrix $P$, we have that $(K_{n},L,c)$ and $(K_{n}, PLP^T, PcP^T)$ have the same feasibility properties, 
where $PLP^{T} = \{PlP^{T} \mid l \in L \}$.

\subsection{The main result}
It will be convenient for now to collapse weak feasibility and weak 
infeasibility into a single status. We say that $(K_n,L,c)$ is in 
\emph{weak status} if it is either weakly feasible or weakly infeasible.
We start with the following basic observation. The proof is left to the readers.
\begin{proposition}\label{prop_non_empty}
If $(K_n,L,c)$ is weakly infeasible, there exists a 
nonzero vector in $K_n \cap L $. 
\end{proposition}

Now we present a key result in our paper.
The following theorem says that if $(K_n,L,c)$ has a hyper feasible direction, then, we can construct
another SDFP of smaller size whose feasibility status is {\it almost} identical to the original problem.

\begin{theorem}\label{theo_decomp}
Let $(K_n,L,c)$ be a SDFP, and consider a subproblem $\pi_k(K_n,L,c)$
for some $k > 0$. If the subproblem $\pi_k(K_n,L,c)$ admits an interior hyper feasible direction (i.e, $\interior \pi_k(K_n)\cap \pi_k(L) \neq \emptyset $) then:
\begin{enumerate}
\item $(K_n, L, c)$ is strongly feasible if and only if $\overline\pi_{k}(K_n,L,c)$ is. 
\item $(K_n, L, c)$ is strongly infeasible if and only if $\overline \pi_{k}(K_n,L,c)$ is.
\item $(K_n, L, c)$ is in weak status if and only if $\overline \pi_{k}(K_n,L,c)$ is.
\end{enumerate}

\end{theorem}\begin{proof}

Due to the assumption, there exists a $n\times n$ matrix
\[
x = \left(\begin{array}{cc} A & 0 \\ 0 & 0 \end{array}\right)
\]
where $A$ is a $k\times k$ positive definite matrix.

We now prove items $1$ and $2$. Item  $3$ will follow by elimination.

$(1) \Rightarrow )$
If $y \in L+c$ is positive definite, all its principal submatrices are also 
positive definite. Therefore, $\overline \pi _{k}(y)$ is positive definite.

$(1) \Leftarrow )$
Suppose that $y \in L +c$ is such that $\overline \pi _{k}(y) \in \interior K_{n-k}$. Then, 
we may write $y = \left(\begin{smallmatrix}	F & E \\E^T & G\end{smallmatrix} \right)$, where $G$ is $(n-k)\times(n-k)$ and 
positive definite. For large and positive  $\alpha$, 
$F +\alpha A$ is positive definite and the Schur complement of 
$y +x\alpha$ is $G - E^T(F +\alpha A)^{-1}E $. Since $G$ is positive 
definite, it is clear that, increasing $\alpha$ if necessary, the Schur complement is also 
positive definite. For such an $\alpha$, $ y +x\alpha \in (L+c)\cap \interior K_n$.

$(2) \Rightarrow)$.
Suppose $(K_n,L,c)$ strongly infeasible. Then there exists $s \in K_n$ such 
that $s\in L^\perp$ and $\inProd{s}{c} = -1$. As $x \in L$, we have $s \in K_n \cap \{x\}^\perp$.
This means that $s$ can be written as $\left( \begin{smallmatrix} 0 & 0 \\ 0 & D\end{smallmatrix}  \right)$, 
where $D$ belongs to $K_{n-k}$. It follows that $\overline \pi _{k}(s) \in \overline \pi _{k}(L)^\perp$ and 
$\inProd{\overline \pi _{ k}(s)}{\overline\pi _{ k}(c)} = -1$. By item $iii.$ of Proposition \ref{prop_certificate},
$\overline \pi _k(K_n,L,c)$ is strongly infeasible.

$(2) \Leftarrow)$.
Now, suppose $\overline \pi _k(K_n,L,c)$ is strongly infeasible.
Note that $\overline \pi _k$ is a non-expansive map, i.e,
$\norm {\overline \pi _k (y) - \overline \pi _k (z)} \leq \norm {y - z}$ holds. In particular, if $\inf _{y \in L+c, z \in K_n} \norm {\overline \pi _k(y) - \overline \pi _k (z)} > 0$, 
then the same is true for $\inf _{y \in L+c, z \in K_n} \norm {y - z} $.
\end{proof}
\subsection{Forward Procedure}

Assume that $(K_n, L, c)$ admits a hyper feasible direction $\widetilde A_1$ of rank $k_1$.
Theorem \ref{theo_decomp} might not be directly applicable but 
after appropriate congruence transformation by a nonsingular matrix $P_1$, we have that
$(K_n, P_1^T L P_1, P_1^T c P_1)$ admits a hyper feasible direction of the form
\[  
A_1 = \left(\begin{array}{cc} \widehat A_1 & 0 \\ 0 & 0 \end{array}\right) = P_1^T\widetilde A _{1} P_1,
\]
where $\widehat A_1$ is a $k_1\times k_1$ positive definite matrix.
The feasibility status of $(K_n, L, c)$ and 
\[
(K_{n-k_1}, \overline\pi_{k_1}(P_1^T L P_1), \overline\pi_{k_1}(P_1^T c P_1))
\]
are mostly the same in the sense that items $1-3$ of Theorem \ref{theo_decomp} hold.  

Now, suppose that $(K_{n-k_1}, \overline\pi_{k_1}(P_1^T L P_1), \overline\pi_{k_1}(P_1^T c P_1))$ admits 
a hyper feasible direction $\widetilde A_2$ of rank $k_2$.  
Then, after appropriate congruence transformation by $\widetilde P_2$,
we obtain that 
\[
(K_{n-k_1}, {\widetilde P_2}^T\bar\pi_{k_1}(P_1^T L P_1)  {\widetilde P_2}, {\widetilde P_2}^T\bar\pi_{k_1}(P_1^T c P_1)  {\widetilde P_2})
\] admits a hyper feasible direction of the form
\[  
\left(\begin{array}{cc} \widehat A_2 & 0 \\ 0 & 0 \end{array}\right),
\]
where $\widehat A_2$ is $k_2\times k_2$ positive definite matrix.

Now, the feasibility status of $(K_{n-k_1}, \overline\pi_{k_1}(P_1^T L P_1), \overline\pi_{k_1}(P_1^T c P_1))$
and 
\[
(K_{n-k_1-k_2}, \overline\pi_{k_2}( {\widetilde P_2}^T\overline\pi_{k_1}(P_1^T L P_1) {\widetilde P_2}), \overline\pi_{k_2} ({\widetilde P_2}^T\overline\pi_{k_1}(P_1^T L P_1) {\widetilde P_2}))
\]
are mostly the same. 
Note that instead of applying  a congruence 
transformation by ${\widetilde P_2}$ to \linebreak $(K_{n-k_1}, \overline\pi_{k_1}(P_1^T L P_1), \overline\pi_{k_1}(P_1^T c P_1))$, 
we can apply a congruence transformation by
\[
P_2 = \left(\begin{array}{cc} I_{k_1} & 0 \\ 0 & \widetilde P_2 \end{array}\right)
\]
to the original problem $(K_n, P_1^T L P_1, P_1^T c P_1)$, i.e., we consider
\[
\left(K_n, P_2^T P_1^T L P_1 P_2,
P_2^T P_1^T c P_1 P_2
 \right)
\]
Then the subproblem defined by the $(n-k_1)\times(n-k_1)$ lower right block matrix is precisely 
\[
(K_{n-k_1}, \widetilde{P}_2^T\bar\pi_{k_1}(P_1^T L P_1)  \widetilde{P}_2, \widetilde{P}_2^T\bar\pi_{k_1}(P_1^T c P_1)  \widetilde{P}_2),
\]
and we may pick $A_2 \in P_2^TP_1^T L P_1P_2$ such that
\[
\overline\pi_{k_1+k_2}(A_2) =\left(\begin{array}{cc} \widehat A_2 & 0 \\ 0 & 0 \end{array}\right).
\]
Note that $A_2$ has the following shape
\[
A_2 = \left(\begin{array}{ccc} * & * & * \\ * & \widehat A_2 & 0 \\ * & 0 & 0 \end{array}\right).
\]

Generalizing the process outlined above, we obtain the following procedure, which 
we call ``forward procedure''. The set of matrices $\{A_1, \ldots, A_m\}$ obtained in this way will be called a 
\emph{hyper feasible partition}.
After each application of 
Theorem \ref{theo_decomp},  the size of the matrices  is reduced at least by one. 
This means that after at most $n$ iterations, a subproblem with no nonzero hyper feasible directions is found.
At this point, no further directions can be added and we will say that the partition is \emph{maximal}.

We note that the problem of checking whether a SDFP  $(K_{\tilde n}, \widetilde L, \tilde c)$
has a nonzero  hyper-feasible direction lies in $\NP\cap \coNP$, in the real computation model.
In fact, by Gordan's Theorem, $(K_{\tilde n}, \widetilde L, \tilde c)$ does not 
have a nonzero hyper-feasible direction if and only if $(K_{n}, {\widetilde L}^{\perp}, 0)$ 
is strongly feasible.

\noindent
{\bf [Procedure FP]}

{\bf Input:} $(K_n,L,c)$

{\bf Output:} a non-singular matrix $P$, a sequence ${k_1, \ldots , k_m }$ and a maximal hyper feasible partition $\{A_1, \ldots , A_m\}$ contained in $P^TLP$. 
The $A_i$ are such that $A_1 = \left(\begin{smallmatrix}\widehat {A}_1 & 0 \\ 0  & 0 \end{smallmatrix} \right)$, 
$A_2 =  \left(\begin{smallmatrix} * & * & * \\ *  & \widehat {A}_2 & 0 \\ * & 0 & 0 \end{smallmatrix} \right)$, 
$A_3 =  \left(\begin{smallmatrix} * & * & * & *\\ *  & * & * & * \\ * & * & \widehat {A}_3 & 0 \\  *  & * & 0 & 0 \end{smallmatrix} \right)$ and 
so forth, where $\widehat{A}_i$ is positive definite and lies in $K_{k_i}$, for every $i$.
\begin{enumerate}
\item Set $i := 1$, $\widetilde{L} := L$, $\widetilde{c} := c$ $K := K_n$, $P := I_n$. 
\item Find (i) $\widetilde{A}_i \in \widetilde L \cap K$, ${\rm tr}(\widetilde{A}_i)=1$ or (ii) $\widetilde B \in \widetilde L^\perp \cap {\rm int} K$, ${\rm tr}(\widetilde{B})=1$.
(Exactly one of (i) and (ii) is solvable.)
If (ii) is solvable, then stop. (No nonzero hyper-feasible direction exists.)
\item Compute a non-singular $\widetilde{P}$ such that, 
\[
\widetilde{P}^T \widetilde{A}_i \widetilde{P} = \left(\begin{array}{cc} \widehat A _i& 0 \\ 0 & 0 \end{array}\right) 
\]
where  $\hat A _i$ is a  positive definite matrix. Let $k_i := \text{rank}(\widetilde{A}_i)$.
\item Compute $M = \begin{pmatrix}
I_{k_1 + \ldots + k_{i-1}} & 0 \\
0 & \widetilde{P}  \end{pmatrix}\;$ and set $P^T := M^TP^T$. (If $i = 1$, take $M = \widetilde{P}$)
\item Let $A_i$ be any matrix in $P^TL P $ such that $\overline \pi_{k_1 + \ldots + k_{i-1}}(A_i) = \widetilde{P}^T\widetilde{A}_i\widetilde{P}$.
	 For each $1 \leq j < i$ exchange $A_{j}$ for $M^TA_{j}M$.
\item Set $\widetilde L:= \overline\pi_{k_{i}}(\widetilde{P}^T\widetilde L \widetilde{P}), \widetilde c:= \overline\pi_{k_{i}}(\widetilde{P}^T\widetilde c \widetilde{P})$,  $K := \overline\pi_{k_{i}}(K_n)$, $i := i+1$ and return to Step 2.
(This step is just to pick the lower-right block after the congruence transformation.)
\end{enumerate}

\begin{proposition}\label{prop_last_problem}
Suppose that $(K_n,L,c)$ is such that there is a nonzero element 
in $K_n \cap L$. Applying \textbf{FP} to $(K_n,L,c)$ we have 
that:
\begin{enumerate}
\item $(K_n, L, c)$ is strongly feasible if and only if $\overline{\pi}_{k_1 + \ldots + k_m}(K_n,P^TLP,P^TcP)$ is. 
\item $(K_n, L, c)$ is strongly infeasible if and only if $\overline{\pi}_{k_1 + \ldots + k_m}(K_n,P^TLP,P^TcP)$ is.
\item $(K_n, L, c)$ is in weak status if and only if $\overline{\pi}_{k_1 + \ldots + k_m}(K_n,P^TLP,P^TcP)$ is weakly feasible.
\end{enumerate}
\end{proposition}
\begin{proof}
If $m = 0$, then the proposition follows because $\overline{\pi}_{0}$ is equal 
to the identity map. In the case $m = 1$, the result  follows from Theorem 
\ref{theo_decomp}.

Note that at the $i$-th iteration, if a direction $A_i$ is found then, after 
applying the congruence transformation $\widetilde{P}$,  
$\overline\pi_{k_{i}}(K,\widetilde{P}^T\widetilde L \widetilde{P},\widetilde{P}^T\widetilde c \widetilde{P})$ preserves 
feasibility properties in the sense of Theorem $1$. Note that it is a SDFP 
over $\mathbb{S}_{n -k_1 - \ldots - k_i}$. Also,  due to the way 
$M$ is selected, we  have that equation
$\overline\pi_{k_{i}}(K,\widetilde{P}^T\widetilde L \widetilde{P},\widetilde{P}^T\widetilde c \widetilde{P})  = 
\overline\pi_{k_1 + \ldots + k_{i}}(K_n,P^TLP,P^TcP)$ holds after Line 4 and before $\widetilde{L}$ and $K$ are updated. This justifies 
items $1.$ and $2.$.

Consider the case where $(K_n,L,c)$ is in weak status. When $(K,\widetilde L , \widetilde c)$ is weakly infeasible 
we can always find a new direction $A_i$ and the size of problem decreases by a positive 
amount, so that $(K,\widetilde L , \widetilde c)$ cannot be weakly infeasible for all iterations. The only 
other possibility is weak feasibility, which justifies item $3$. 
\end{proof}

The matrices $A_1,  \ldots, A_m $ obtained through \textbf{FP} have the shape
\[
\left(\begin{array}{cccc}\widehat A_1 & 0 & 0 & 0\\
0 & 0 & 0 & 0\\
0 & 0 & 0 & 0\\
0 & 0 & 0 & 0
\end{array}\right),\ \ \
\left(\begin{array}{cccc}* & * & * & *\\ * &\widehat A_2 & 0 & 0 \\ * & 0 & 0 & 0 \\ * & 0 & 0 & 0\end{array}\right),\ \ \ 
\left(\begin{array}{cccc}* & * & * & * \\ * &* & * & * \\ * & * & \widehat A_3 & 0 \\ * & *& 0 & 0\end{array}\right), \ldots
\]
where $\widehat A_1, \widehat A_2, \widehat A_3, \ldots $ are positive definite.
The matrix $A_i$ are referred to as {\it sub-hyper feasible directions}, since the 
$\widehat A_i$ are hyper feasible directions. 
The problem  $\overline{\pi}_{k_1 + \ldots + k_m}(K_n,P^TLP,P^TcP)$ will be
referred to as the \emph{last subproblem} of $(K_n,L,c)$.

\begin{example}\label{example_partition}
Let  
\begin{equation}\label{equation_example_partition}
	L+c = \left \{ 
	\begin{pmatrix}
		t & v & 1 & u\\ 
		v & z+2  & v+1 & z+1 \\
		1 & v+1 & u-1 & s \\
		u & z+1 & s & 0
	\end{pmatrix}
	 \mid t,u,v,s,z \in \mathbb{R} \right \}.
\end{equation}
and let us apply \textbf{FP} to $(K_4,L,c)$. The first direction can be, 
for instance, $A_1 = \left(\begin{smallmatrix}		1 & 0 & 0 & 0\\
		0 & 0 & 0 & 0\\
		0 & 0 & 0 & 0\\
		0 & 0 & 0 & 0
	\end{smallmatrix}\right)$. Then $k_1 = 1$ and $\widetilde{P}$ is the identity, at this step. 
At next iteration, we have $K = K_3$ and $\widetilde {L} = \left \{ 
	\left(\begin{smallmatrix}	
		z & v & z \\
		v & u & s \\
		z & s & 0
	\end{smallmatrix}\right)
	 \mid u,s,v,z \in \mathbb{R} \right \}.$ Then, $\widetilde {A}_2$ can be taken 
	 as $ \left(\begin{smallmatrix}	
		0 & 0 & 0 \\
		0 & 1 & 0 \\
		0 & 0 & 0
	\end{smallmatrix}\right)$ and $k_2$ is $1$. A possible choice of $\widetilde {P}$ is 
	$ \left(\begin{smallmatrix}	
		0 & 1 & 0 \\
		1 & 0 & 0 \\
		0 & 0 & 1
	\end{smallmatrix}\right)$. Then $P$ is $\left(\begin{smallmatrix}		1 & 0 & 0 & 0\\
		0 & 0 & 1 & 0\\
		0 & 1 & 0 & 0\\
		0 & 0 & 0 & 1
	\end{smallmatrix}\right)$ and we can take $A_2 = \left(\begin{smallmatrix}		0 & 0 & 0 & 1\\
		0 & 1 & 0 & 0\\
		0 & 0 & 0 & 0\\
		1 & 0 & 0 & 0
	\end{smallmatrix}\right)$. $\tilde{L}$ is then updated and it becomes  $\left \{ 
	\left(\begin{smallmatrix}	
		z & z \\
		z & 0
	\end{smallmatrix}\right)
	 \mid z \in \mathbb{R} \right \}.$ The procedure stops here, because $0$ is the only positive 
	 semidefinite matrix in $\tilde{L}$. 

Now, $\overline{\pi}_2(P^T(L+c)P)$ is $\left \{ 
	\left(\begin{smallmatrix}	
		z +2 & z + 1 \\
		z + 1 & 0
	\end{smallmatrix}\right)
	 \mid z \in \mathbb{R} \right \}$, so $\overline{\pi}_2(K_4,P^TLP,P^TcP)$ is a weakly feasible 
	 system. Therefore, by Proposition \ref{prop_last_problem}, $(K_4,L,c)$ has 
	 weak status and is either weakly infeasible or weakly feasible. The 
	 $0$ in the lower right corner of \eqref{equation_example_partition} 
	 forces $u = 0$, $z = -1$ and $s = 0$, but 
	 this assignment produces a negative element in the diagonal. This tells 
	 us that $(K_4,L,c)$ is infeasible so it must be weakly infeasible.
\end{example}	

\begin{corollary} \label{Certificate_Weak}
The matrices $A_1, \ldots , A_m, P, \widetilde B$ as described in \textbf{FP} together with a 
 finite weak feasibility certificate for $\overline{\pi}_{k_1 + \ldots k_m}(K_n,P^TLP,P^TcP)$ form a 
finite certificate that $(K_n,L,c)$ is in weak status. If no such a certificate exists, then either 
 item $1$ or item $3$ of Proposition \ref{prop_certificate} holds. This shows that deciding whether a 
SDFP is in weak status is in $\NP \cap \coNP$.

\end{corollary}
\begin{proof}
Follows directly from Proposition \ref{prop_last_problem}.
\end{proof}

\subsection{Maximum number of directions 
required to approach the positive semidefinite cone}
According to Proposition \ref{prop_non_empty}, there is always 
a nonzero element in $K_n\cap L$ when $(K_n,L,c)$ is weakly infeasible.
Therefore, a natural question is, 
given a weakly infeasible $(K_n,L,c)$, whether it is always possible 
to select a point in $x \in L+c$ and then a nonzero direction $d \in K_n\cap L$ such 
that $\lim _{t \to +\infty } \text{dist}(x +td, K_n) = 0$ or not.  
We call weakly infeasible problems 
having this property \emph{directionally weakly infeasible} (DWI).
The simplest  instance of DWI problem is
\[
\max \; 0\ \text{s.t.} \; \begin{pmatrix}
		t & 1 \\
		1 & 0
	 \end{pmatrix} \in K_2,\ t\in \mathbb{R}.
\] 
Unfortunately, not all weakly infeasible problems are DWI, as shown in the
following instance.

\begin{example}[A weakly infeasible problem that is not directionally weakly infeasible]\label{example_not_dwi}
Let $(K_3,L,c)$ be such that $L+c = \left \{ 
	\left(\begin{smallmatrix}
		t & 1 & s \\
		1 & s & 1 \\
		s & 1 & 0
	\end{smallmatrix}\right)  \mid t,s \in \mathbb{R}\right\} $ and let 
$A_1 = \left(\begin{smallmatrix}
		1 & 0 & 0 \\
		0 & 0 & 0 \\
		0 & 0 & 0
	\end{smallmatrix}\right)$.

Applying Theorem 1 twice, we see that the problem is in weak status.
Looking at its $2\times 2$ lower right block, we see this problem is infeasible
and hence is weakly infeasible.  But this problem is not DWI.
If $(K_{3},L,c)$ were DWI, 
we would have $\lim _{t \to +\infty}\mathrm{dist}(tA_1 + c', K_3) = 0$,
for some $c' \in L+c$. To show this does not hold, we fix $s$.  
Regardless of the value of $t\geq 0$, the minimum eigenvalue of the matrix is
uniformly negative, since its $2\times 2$ lower right block is strongly infeasible.
\end{example}

Thus, a weakly infeasible problem is not DWI in general.
If we let $s$ sufficiently large in the example, 
then the minimum eigenvalue of the lower ${2\times 2}$
matrix gets very close to zero.  This will make $(1,3)$ and $(3,1)$ elements large.
But we can let $t$ much larger than $s$.  Then, the minimum eigenvalue of the submatrix
$\left(\begin{smallmatrix}
		t & s \\
		s & 0
	\end{smallmatrix}\right) $
is close to zero. Intuitively, this neutralize the effect of big off-diagonal elements,
and we obtain points in $L+c$ arbitrarily close to $K_3$, by taking $s$ to be large and $t$ to be much
larger than $s$.  

Generalizing this intuition, in the following, we show that $n-1$ directions are enough 
to approach the positive semidefinite cone.
First we discuss how the hyper feasible partition $\{A_1,\ldots , A_m \}$ of \textbf{FP} 
fits in the concept of tangent cone.
We recall that for $x \in K_n$ the cone of feasible directions is 
the set $\dirCone(x,K_n) = \{d \in \mathbb{S}_n \mid \exists t > 0 \text{ s.t } x +td \in K_n \}$.
Then the tangent cone at $x$ is the closure of $\dirCone (x,K_n)$ and is 
denoted by $\tanCone (x,K_n)$. It can be shown that if $d \in \tanCone (x,K_n)$ then $\lim _{t\to +\infty} \text{dist}(tx +d,K_n) = 0$.

We remark that if $x = \left(\begin{smallmatrix}	D & 0 \\ 0 & 0  \end{smallmatrix}\right)$, where 
$D$ is positive definite $k \times k$ matrix, then $\tanCone (x,K_n)$ consists of all 
symmetric matrices $\left(\begin{smallmatrix}	* & * \\ * & E  \end{smallmatrix}\right)$, where 
$*$ denotes arbitrary entries and $E$ is a positive semidefinite $(n-k)\times (n-k)$ matrix.
See \cite{pataki_handbook} for more details.

The output $\{A_1,\ldots , A_m \}$ of \textbf{FP} 
is such that $A_2 \in  \tanCone(A_1,K_{n} )$. This 
is clear from the shape of $A_1$ and $A_2$, and from a simple 
argument using the Schur Complement.  Now, $A_3$ is such 
that $\overline{\pi} _{k_1+k_2}(A_3)$ is positive semidefinite. 
We have 
$A_2 = \left(\begin{smallmatrix} * & * & * & *\\ *  & \widehat {A}_2 & 0 & 0 \\ * & 0 & 0  & 0 \\  *  & 0 & 0 & 0 \end{smallmatrix} \right)$
$A_3 = \left(\begin{smallmatrix} * & * & * & *\\ *  & * & * & * \\ * & * & \widehat {A}_3 & 0 \\  *  & * & 0 & 0 \end{smallmatrix} \right)$.
Then $   \left(\begin{smallmatrix}  * & * &  * \\ * & \widehat {A}_3 & 0  \\  *  & 0 & 0  \end{smallmatrix} \right) \in \tanCone \left( \left(\begin{smallmatrix}  \widehat {A}_2 & 0 &  0 \\ 0 & 0 & 0  \\  0  & 0 & 0  \end{smallmatrix} \right), K_{n-k_1} \right)$, 
i.e, $\overline{\pi} _{k_1}(A_3) \in  \tanCone(\overline{\pi} _{k_{1}}(A_{2}),K_{n-k_1} )$. Denote
$k_1 + \ldots + k_i$ by $N_i$ and set $N_0 = 0$. Then, for $i > 2$, we have:
\begin{align*}
\overline{\pi} _{N_{i-2}}(A_i) \in  \tanCone(\overline{\pi} _{N_{i-2}}(A_{i-1}),K_{n-N_{i-2}} ).
\end{align*}

Moreover, if the last subproblem $\overline{\pi}_{N_m}(K_n,L,c)$ has a feasible solution, 
we can pick some $c'$ such that $\overline{\pi}_{N_m}(c')$ is positive 
semidefinite. Then $\overline{\pi}_{N_{m-1}}(c') \in \tanCone(\overline{\pi} _{N_{m-1}}(A_{m}),K_{n-N_{m-1}} )$.
Given $\epsilon > 0$, by picking $\alpha _m > 0$ sufficiently large we have $\text{dist} (\overline{\pi}_{N_{m-1}}(c'+\alpha _m A_m ), K_{n-N_{m-1}}) < \epsilon $.
Now, $\overline{\pi}_{N_{m-2}}(x+\alpha _m A_m )$ does not necessarily lie on the 
tangent cone of $\overline{\pi}_{N_{m-2}}(A_{m-1})$ at $K_{n-N_{m-2}}$, but still 
it is possible to pick $\alpha _{m-1} > 0$ such that 
\begin{equation*}
\text{dist} (\overline{\pi}_{N_{m-2}}(c'+\alpha _m A_m +\alpha _{m-1} A_{m-1} ), K_{n-N_{m-2}}) < 2\epsilon.
\end{equation*}
In order to show this, let $h \in K_{n-N_{m-1}}$ be such that 
\begin{equation*}
\norm{\overline{\pi}_{N_{m-1}}(c'+\alpha _m A_m) -h} = \text{dist} (\overline{\pi}_{N_{m-1}}(c'+\alpha _m A_m ), K_{n-N_{m-1}}).
\end{equation*}
Now, define $\widetilde h$ to be the matrix $\overline{\pi}_{N_{m-2}}(c'+\alpha _m A_m)$, except 
that the lower right $(n-k_m)\times (n-k_m)$ block is replaced by $h$. It follows readily 
that $\widetilde h$ lies on the tangent cone of $\overline{\pi}_{N_{m-2}}(A_{m-1})$. Then, 
we may pick $\alpha _{m-1} > 0$ sufficiently large such that  $\text{dist} (\overline{\pi}_{N_{m-2}}(\alpha _m A_m) + h, K_{n-N_{m-2}}) < \epsilon $.
Let $y_1 = \overline{\pi}_{N_{m-2}}(c'+\alpha _m A_m)$, $y_2 =  \overline{\pi}_{N_{m-2}}(\alpha _{m-1} A_{m-1})$. 
We then have the following implications:
\begin{align*}
\text{dist} (y_1 + y_2, K_{n-N_{m-2}}) & \leq \text{dist} (y_1 - \widetilde h, K_{n-N_{m-2}}) + \text{dist} (y_2 +\widetilde h, K_{n-N_{m-2}}) \\
& \leq \norm{\overline{\pi}_{N_{m-1}}(c'+\alpha _m A_m) - h} + \epsilon  \leq 2\epsilon.
\end{align*}

If we continue in this way, it becomes clear that $\alpha _1, \ldots , \alpha _m$ can be selected 
such that $\text{dist} (c'+\alpha _m A_m +\alpha _{m-1} A_{m-1} + \ldots + \alpha _1 A_1, K_{n}) < m\epsilon $.
This shows how the directions $\{A_1,\ldots , A_m \}$ can be used to construct points that are 
arbitrarily close to $K_n$, when the last subproblem is feasible. This leads to the next theorem.

\begin{theorem}\label{theo_weakly_infeasible_subspace}
If $(K_{n},L,c)$ is weakly infeasible then there exists an   affine space of dimension at 
most $n-1$ such that  $L' + c' \subseteq L+c$ and $(K_{n},L',c')$ is weakly infeasible.	
\end{theorem} 
\begin{proof}
The construction above shows that if $L'$ is the space spanned by $\{A_1, \ldots , A_m \}$ 
and $c'$ is taken as above, then $(K_n,L',c')$ is weakly infeasible. As $(K_n,L,c)$ 
is weakly infeasible, we have $m > 0$. We also have $k_1 + \ldots + k_m \leq n$, which 
implies $m \leq n$. Notice that $\overline{\pi} _{n}(K_n,P^TLP,P^TcP)$ is strongly feasible, because 
it is equal to the system $(\{0\}, \{0\},0)$. Therefore $k_1 + \ldots + k_m < n$, which 
forces $m < n$.
\end{proof}

\section{Backward Procedure}\label{sec:backward_procedure}
In this section, we discuss a ``backward procedure'' for distinguishing the 4 different 
feasibility statuses. The main difficulty is when the problem is in 
weak status. In that case, due to Proposition \ref{prop_last_problem}, the last subproblem is  
weakly feasible. This offers the opportunity to shrink both the last subproblem 
and the whole problem, as discussed in our next theorem.

\begin{theorem}\label{theo_backwards}
Let $(K_n,L,c)$ be a given SDFP, satisfying the following assumptions:
\begin{enumerate}
	\item for some $k > 0$, $\overline \pi _{ k}(K_n,L,c)$ is weakly feasible.
	\item for some $l$ such that $0 \leq l < n-k$, the face 	$F = \left\{ \left(\begin{smallmatrix}	A & 0 \\0 & 0 \end{smallmatrix}\right) \mid A \in K_{l} \right\}$ contains the feasible region of $\overline \pi _{k}(K_n,L,c)$.
\end{enumerate}
\begin{equation*}
\left(\begin{array}{c| c }
* &  * \\\hline
* & \overline \pi_k(L+c) \\
\multicolumn{1}{c}{$\upbracefill$} & \multicolumn{1}{c}{$\upbracefill$}   \\
\multicolumn{1}{c}{\scriptstyle k} & \multicolumn{1}{c}{\scriptstyle n-k}  \\

\noalign{\vspace{-2\normalbaselineskip}}
  \end{array} \right)
  =
\left(\begin{array}{c| c c}
* & *  & * \\\hline
* & A  & 0   \\
* & 0 & 0 \\
\multicolumn{1}{c}{$\upbracefill$} & \multicolumn{1}{c}{$\upbracefill$}  & \multicolumn{1}{c}{$\upbracefill$}  \\
\multicolumn{1}{c}{\scriptstyle k} & \multicolumn{1}{c}{\scriptstyle l}  & \multicolumn{1}{c}{\scriptstyle n-k-l}  \\

\noalign{\vspace{-2\normalbaselineskip}}
  \end{array} \right)
  \vspace{2\normalbaselineskip},
\end{equation*}

Furthermore, let $E$ be the set of $(k+l)\times(k+l)$ upper left principal submatrices in $\mathbb{S}_n$, i.e., $E = \left\{ \left(\begin{smallmatrix}	B & 0 \\0 & 0 \end{smallmatrix}\right) \mid B \in \mathbb{S}_{k+l} \right\}$,
and define $\widetilde L$ and $\widetilde c$ as the vector subspace and a vector such that $E\cap (L+c)$ is 
equal to 
$\widetilde L + \widetilde c$, i.e.,  
\begin{equation}
E\cap (L + c) = \widetilde{L} + \widetilde{c} = 
\left(\begin{array}{c c| c}
\multicolumn{2}{c|}{\pi_{k+l}(\widetilde{L} + \widetilde{c})} & 0 \\
  &   & 0   \\\hline
  0 & 0 & 0 \\
\multicolumn{1}{c}{$\upbracefill$} & \multicolumn{1}{c}{$\upbracefill$}  & \multicolumn{1}{c}{$\upbracefill$}  \\
\multicolumn{1}{c}{\scriptstyle k} & \multicolumn{1}{c}{\scriptstyle l}  & \multicolumn{1}{c}{\scriptstyle n-k-l}  \\
\noalign{\vspace{-2\normalbaselineskip}}
  \end{array} \right)
  \vspace{2\normalbaselineskip}. \label{aaa}
\end{equation}
($\widetilde L+\widetilde c$ is an affine subspace of $L + c$ where all but the upper left block is set to zero.)

Then, the following holds: If \eqref{aaa} is empty, $(K_n,L,c)$ is infeasible. Otherwise, $(K_n,L,c)$ is feasible if and  only if the SDFP subproblem 
$\pi_{k+l}(K_n,\widetilde{L},\widetilde{c})$ is feasible.
\end{theorem}
\begin{proof}
Due to assumption 
$1$, $\overline \pi _{k}(K_n,L,c)$ is weakly feasible, so if $x \in L+c$ and 
$\overline \pi _{k}(x) $ is positive semidefinite then $x$ has the format 
$ \left(\begin{smallmatrix}	* & * \\ * & C \end{smallmatrix}\right) $, where 
$C$ is a $(n-k)\times (n-k)$ positive semidefinite definite and $*$ denotes 
arbitrary entries. Now, due to Assumption $2$, $C$ itself has the format 
$\left(\begin{smallmatrix}	A & 0 \\0 & 0 \end{smallmatrix}\right)$, where 
$C$ is a $l\times l$ positive semidefinite matrix . So, actually, $x$ has the format 
$ \left(\begin{smallmatrix}	* & * & * \\ * &  C  & 0 \\ * & 0 & 0 \end{smallmatrix}\right) $.

In order for $x$ itself to be positive semidefinite, the elements in the upper 
right and lower left must be 0. In other words, $x$ must  have 
the format  $ \left(\begin{smallmatrix}	* & * & 0 \\ * &  C  & 0 \\ 0 & 0 & 0 \end{smallmatrix}\right) $, 
where  $ \left(\begin{smallmatrix}	* & *  \\ * & C \end{smallmatrix}\right) $ is 
a $(k+l)\times (k+l)$ positive semidefinite matrix. 
Therefore, if $E \cap (L+c) = \emptyset$ there is no  way $(K_n,L,c)$ could be feasible.

If $E \cap (L+c)$ is not empty, since $\pi _{k+l}(x) =  \left(\begin{smallmatrix}	* & *  \\ * & C \end{smallmatrix}\right) $, 
it is clear that the feasibility of $(K_n,L,c)$ is equivalent to the feasibility 
of $\pi _{k+l}(K_n,\widetilde L,\widetilde c)$.
\end{proof}

We remark that whenever $(K_n,L,c)$ satisfies Assumption $1$ , 
it is possible to apply a congruence transformation to $(K_n,L,c)$ in order to meet 
Assumption $2$. Using Theorem 11.3 of \cite{rockafellar}, the weak feasibility of $\overline \pi _{k}(K_n,L,c)$ 
implies the existence of $w \neq 0$ such that $\inProd{w}{ \overline \pi _{k}(l+c)} \leq \inProd{w}{x}$, 
for every $l \in L$ and $x \in K_{n-k}$. The only way this inequality can hold is 
is if $w \in K_{n-k}\cap \overline \pi _{k}(L)^\perp $ and $\inProd{w}{ \overline \pi _{k}(c)}\leq 0 $. As 
$\overline \pi _{k}(K_n,L,c)$ is not strongly infeasible, we have $\inProd{w}{ \overline \pi _{k}(c)}= 0  $.
Changing $L+c$ and $w$ by a congruence transformation if necessary, we may
assume that $w =  \left(\begin{smallmatrix}	0 & 0 \\ 0 & \widetilde w\end{smallmatrix} \right) $, 
where $ \widetilde w$ is a $(n-k-l)\times (n-k-l)$ positive definite matrix and 
$l < n-k$.
Then, it is clear that Assumption $2$ holds for the transformed problem.

The search for a $w$ as above is essentially one step of a facial 
reduction algorithm \cite{borwein_facial_1981,pataki_strong_2013,article_waki_muramatsu}. Each iteration of a facial reduction algorithm  
aims to find a proper face of $K_n$ that still 
contains the feasible region. Usually, however, the search is done on
the whole problem. 
Let 
$w' =  \left(\begin{smallmatrix}	0 & 0 \\ 0 & w \end{smallmatrix} \right) $, where $w' \in K_n$, 
then $w' \in K_{n}\cap L^\perp \cap \{c \}^\perp$. Which means that $w'$ can be used 
to perform a step of facial reduction on the whole problem. In particular, 
if $x$ is a feasible point, since $\inProd{x}{w} = 0$, it must be true 
that $x = \left(\begin{smallmatrix}	D & 0 \\ 0 & 0 \end{smallmatrix} \right)$, where $D$ 
is a positive semidefinite $(k+l)\times (k+l)$ matrix.
The idea is that the knowledge of a weakly feasible subproblem makes 
it possible to confine the search to a smaller subproblem and still find a
smaller face of $K_n$ that contains the feasible region of the original problem.

If we apply 
Theorems \ref{theo_decomp} and \ref{theo_backwards} repeatedly, we obtain
a facial reduction-like procedure  which is able to determine the  feasibility 
status of a given $(K_n,L,c)$ as shown below.

\noindent
{\bf [Procedure BP]}
\begin{enumerate}[Step 1.]
\item Apply \textbf{FP} to $(K_n,L,c)$.
If the last subproblem $\overline {\pi} _{k_1 + \ldots + k_m}(K_n,P^TLP,P^TcP)$ is strongly infeasible, then 
$(K_n,L,c)$ is also strongly infeasible.
If $\overline {\pi} _{k_1 + \ldots + k_m}(K_n,P^TLP,P^TcP)$  is strongly feasible, then 
$(K_n,L,c)$ is also strongly feasible.
In both cases we stop the procedure.
Otherwise set $i = 0,F_0 = K_n, L_0 = L, c_0 = c$.

\item If we reach this step, $\overline {\pi} _{k_1 + \ldots + k_m}(F_i,P^TL_iP,P^Tc_iP)$, is \emph{weakly feasible}, i.e., 
$(F_i,L_i,c_i)$ is in weak status.
Applying a congruence transformation to $(F_i,P^TL_iP,P^Tc_iP)$, if necessary, both 
assumptions of Theorem \ref{theo_backwards} can be met. Let $K_{k+l} $, $\widetilde L$,
$\widetilde c$ and $E$ be as in Theorem \ref{theo_backwards}.
If $E \cap P^T(L_i+c_i)P$ is empty, we stop and 
declare $(K_n,L,c)$ to be weakly infeasible. Otherwise, we obtain 
$\widetilde L+ \widetilde  c$ such that $E \cap P^T(L_i+c_i)P = \widetilde L+\widetilde c$ and a projection $\pi _{k+l}$.

\item Apply \textbf{FP} to $\pi _{k+l}(F_i,\widetilde L,\widetilde c)$ and obtain a 
new projection $\overline {\pi} _{k_1 + \ldots + k_m}$. 
If $\overline {\pi} _{k_1 + \ldots + k_m}(K_{k+l},\pi _{k+l}(\widetilde L),\pi _{k+l}(\widetilde c))$ is strongly feasible, then $(K_n,L,c)$ is weakly feasible. 
If $\overline {\pi} _{k_1 + \ldots + k_m}(K_{k+l},\pi _{k+l}(\widetilde L),\pi _{k+l}(\widetilde c))$ is strongly infeasible, then $(K_n,L,c)$ is weakly infeasible.
In both cases, we end the procedure. Otherwise, 
set $F_{i+1} := K_{k+l}$, $L_{i+1} := \pi _{k+l}(\widetilde L)$, $c_{i+1} := \pi _{k+l}(\widetilde c)$, $i := i+1$ and 
return to Step $2.$
\end{enumerate}

\noindent{\bf Remark:}  The procedure terminates in at most $n$  iterations, because the size of the problem is reduced at least by one for each 
iteration.

\begin{example}\label{example_bp}
Let $L$ and $c$ be as in Example \ref{example_partition} and let us apply 
\textbf{BP} to $(K_4,L,c)$. At Step 1, we apply 
to \textbf{FP} and we obtain $\overline{\pi}_2(P^T(L+c)P) = 
\left \{ 
	\left(\begin{smallmatrix}	
		z +2 & z + 1 \\
		z + 1 & 0
	\end{smallmatrix}\right)
	 \mid z \in \mathbb{R} \right \}$. And $\overline{\pi}_2(K_4,P^TLP,P^TcP)$ is 
	 weakly feasible, so we move on to Step 2. The  feasible region 
	 of $\overline{\pi}_2(K_4,P^TLP,P^TcP)$ consists of a single matrix which 
	 is $\left(\begin{smallmatrix}	
		1 & 0 \\
		0 & 0
	\end{smallmatrix}\right)$. We are under the conditions of Theorem 
	$\ref{theo_backwards}$ and 
	\begin{equation*}
	E \cap P^T(L+c)P = \left \{ \begin{pmatrix} t & 1 & v & 0 \\ 1 & -1 & v+1 & 0 \\ v & v+1 & 1 & 0 \\ 0 & 0 & 0 & 0 \end{pmatrix}  \mid t,v \in \mathbb{R} \right \}.
	\end{equation*}
Then $\pi _3 (\widetilde{L} + \widetilde{c}) =  \left \{ \left(\begin{smallmatrix}	 t & 1 & v  \\ 1 & -1 & v+1  \\ v & v+1 & 1 \end{smallmatrix}\right) \mid t,v \in \mathbb{R} \right \}$.
Applying \textbf{FP} to $ (K_3,\pi _M(\widetilde{L}),\pi _M(\widetilde{c}))$ we obtain as output 
$P = I_3 $, $m = 1$, $k_1 = 1$ and $A_1 = \left(\begin{smallmatrix}	1 & 0 & 0 \\ 0 & 0 & 0 \\ 0 & 0 & 0 \end{smallmatrix}\right)$. Then 
$\overline {\pi} _1(\pi _3 (\widetilde{L} + \widetilde{c})  )  =  \left \{ \left(\begin{smallmatrix}	 -1 & v+1  \\  v+1  & 1 \end{smallmatrix}\right) \mid v \in \mathbb{R} \right \}$.
The $-1$ in the upper left entry ensures that the system  $(K_2,\overline {\pi} _1(\pi _3 (\widetilde{L})),\overline {\pi} _1(\pi _3 (\widetilde{L})))   $ is strongly infeasible, which 
shows that $(K_4,L,c)$ is weakly infeasible itself.

\end{example}

\subsection{Complexity aspects of \textbf{BP} }
Let us discuss briefly  certificates and complexity issues regarding \textbf{BP}.
Both \textbf{BP} and \textbf{FP} can be thought as procedures that invoke several oracles. 
For instance, we can consider that a nonzero hyper feasible direction, as required in \textbf{FP}, is 
obtained by querying an oracle.
According to the recipe explained in the second paragraph of Section 2.1, we can show that
the  problem of deciding the feasibility status  of $(K_n, L, c)$ has a finite certificate and 
that \textbf{BP} acts as as verifier procedure.  All we have to do is argue that 
all the computations required by \textbf{BP} can be checked in polynomial time.

First note that all computations done by \textbf{FP} can be checked 
either by the certificates discussed in Proposition \ref{prop_certificate} or 
by Gordan's Theorem. The same is true for 
Steps 1 and 3 of \textbf{BP}. The only part of \textbf{BP} that needs further analysis 
is when Theorem \ref{theo_backwards} is invoked at Step 2, where
we need to check that assumption 2 of Theorem \ref{theo_backwards} holds. 
However, we can use as certificate a nonzero element $w$ satisfying $w \in F_i \cap (P^TL_iP)^\perp \cap \{P^Tc_iP \}^\perp$, as in the discussion 
that follows the proof of Theorem \ref{theo_backwards} and, if necessary, a non-singular
matrix which puts the problem in the correct shape.

This provides an alternative proof of the fact 
that for each different feasibility status, the problem of deciding 
whether $(K_n, L, c)$ has that status is in  $\NP\cap \coNP$ in the BBS model of 
real computation.

\section{Conclusion}\label{sec:conc}
In this article we presented an analysis of weakly infeasible problems via 
two procedures: \textbf{FP} and \textbf{BP}. The procedure \textbf{FP} produces 
as an output a finite set of directions and for weakly infeasible problem, 
they can be used to construct $L+c$ arbitrarily close to $K_n$. 
The procedure \textbf{BP} uses 
\textbf{FP} and is able to distinguish between the four feasibility statuses. The computations
involved in both procedures might be hard, but they are verifiable in polynomial time, in the 
BSS model.
Extension of our analysis to blockwise SDPs and to other classes of conic linear programs is an 
interesting topic for future research.

\section*{Acknowledgments.}
M. Muramatsu and T. Tsuchiya are supported in part with Grant-in-Aid
for Scientific Research (B)24310112 and (C)2630025.
M. Muramatsu is also supported in part with Grant-in-Aid for Scientific Research (B) 26280005.

\bibliographystyle{plain}
\bibliography{bib}
\end{document}